\documentclass[reqno, 12pt]{amsart}
\usepackage{array}
\usepackage{amsmath}
\usepackage{amsfonts}
\usepackage{amssymb}
\usepackage{mathrsfs}
\usepackage{enumerate}
\usepackage{amsthm}
\usepackage{verbatim}
\usepackage{amsmath, amscd}
\usepackage[usenames,dvipsnames]{color}
\usepackage{bm}
\usepackage{xy}
\xyoption{all}
\usepackage{color}
\usepackage{marvosym}
\usepackage{amsbsy}
\usepackage{hyperref}
\usepackage{tikz}
\usepackage{marginnote}
\usetikzlibrary{matrix,arrows,backgrounds}
	
\newtheorem{theorem}{Theorem}[section]

\newtheorem{lemma}[theorem]{Lemma}
\newtheorem{proposition}[theorem]{Proposition}
\newtheorem{corollary}[theorem]{Corollary}

\newtheorem{conjecture}[theorem]{Conjecture}
\newtheorem{question}[theorem]{Question}

\theoremstyle{definition}
\newtheorem{definition}[theorem]{Definition}
\newtheorem{remark}[theorem]{Remark}

\newtheorem{example}[theorem]{Example}

\newcommand{\op}[1]{\operatorname{#1}}

\newcommand{\newterm}{\textsf}

\newcommand{\dbcoh}[1]{\operatorname{D}^{\operatorname{b}}(\operatorname{coh }#1)}

\newcommand{\dsing}[1]{\operatorname{D}_{\operatorname{sg}}(#1)}

\newcommand{\gm}{\mathbb{G}_m}

\def\N{\op{\mathbb{N}}}
\def\Z{\op{\mathbb{Z}}}

\def\R{\op{\mathbb{R}}}
\def\Q{\op{\mathbb{Q}}}

\def\O{\op{\mathcal{O}}}
\def\A{\op{\mathbb{A}}}
\def\P{\op{\mathbf{P}}}

\def\T{\op{\mathcal{T}}}

\def\tif{\text{if } }

\def\tand{\text{ and } }

\title[The Batyrev-Nill Conjecture]{Proof of a conjecture of Batyrev and Nill}

\author[Favero]{David Favero}
\address{
  \begin{tabular}{l}
   David Favero \\
   \hspace{.1in} University of Alberta, Department of Mathematics,  Edmonton, AB Canada \\
   \hspace{.1in} Email: {\bf favero@ualberta.ca} \\
  \end{tabular}
}

\author[Kelly]{Tyler L. Kelly}
\address{
  \begin{tabular}{l}
   Tyler L. Kelly \\ 
   \hspace{.1in} University of Cambridge, Department of Pure Mathematics and Mathematical \\ \hspace{.1in} Statistics,  Wilberforce Road, Cambridge, United Kingdom CB3 0WB \\
   \hspace{.1in} Email: {\bf tlk20@dpmms.cam.ac.uk} \\
  \end{tabular}
}

\numberwithin{equation}{section}
\begin{document}

\begin{abstract}
We prove equivalences of derived categories for the various mirrors in the Batyrev-Borisov construction.    In particular, we obtain a positive answer to a conjecture of Batyrev and Nill. The proof involves passing to an associated category of singularities and toric variation of geometric invariant theory quotients.
\end{abstract}

\maketitle

 \section{Introduction}

Witten proposed that two $2$-dimensional topological field theories are derived from a conformal field theory with target space a Calabi-Yau manifold $\mathcal{M}$: the $A$-model and the $B$-model \cite{Witten}.  The $A$-model depends only on the symplectic structure of $\mathcal{M}$ and the $B$-model only on the complex structure.  Mirror symmetry is a duality which proposes the existence of a mirror Calabi-Yau manifold $\mathcal{W}$ on which the $A$-model and $B$-model are transposed from those corresponding to $\mathcal{M}$.
  
In 1994, Kontsevich described a categorical approach to this exchange of symplectic and complex structures \cite{Kontsevich}. Kontsevich articulated that given a mirror pair of Calabi-Yau manifolds $\mathcal{M}$ and $\mathcal{W}$, the Fukaya category of $\mathcal{M}$ (the A-model) should be equivalent to the bounded derived category of coherent sheaves of its mirror $\mathcal{W}$ (the B-model) and vice versa, i.e.,
\[
\op{Fuk}(\mathcal{M}) \cong \dbcoh{\mathcal{W}} \tand \op{Fuk}(\mathcal{W}) \cong \dbcoh{\mathcal{M}} .
\]
This deep conjecture known as Homological Mirror Symmetry has motivated a myriad of geometric research in the last twenty years, extending from results in theoretical physics and algebraic and symplectic geometry, to category and number theory.  

As a consequence of the Homological Mirror Symmetry Conjecture, the derived category of a mirror should not depend  on the construction of the mirror.  If we have multiple mirrors $\mathcal W_1, \ldots, \mathcal W_r$ that arise from various choices of mirror construction on $\mathcal M$, then we expect that these mirrors have equivalent derived categories, i.e., 
\[
\op{Fuk}(\mathcal{M}) \cong  \dbcoh{\mathcal{W}_1} \cong \ldots \cong  \dbcoh{\mathcal{W}_r}.
\]

Even in the foundational mirror construction of Batyrev and Borisov \cite{BB96} for Calabi-Yau complete intersections in toric varieties, one may obtain non-isomorphic mirrors by making different combinatorial choices.  In this paper, we prove that these distinct mirrors have equivalent derived categories of coherent sheaves, fulfilling Kontsevich's prediction for multiple mirrors. This is a positive answer to a conjecture of Batyrev and Nill \cite{BN07}. The central technique, following \cite{HHP, HW}, is to compare geometric invariant theory quotients for the Cox construction of the total space of the vector bundle associated to the complete intersection.  We now provide a precise mathematical explanation of our results. 

\subsection{Results}
 Let $\kappa$ be an algebraically closed field of characteristic zero. Take a lattice $M$ of dimension $d$ with dual lattice $N$. Given a reflexive polytope $\Delta \subseteq M_{\R}:= M \otimes_{\Z} \R$, one associates a Fano toric variety $\mathbb{P}_{\Delta}$. Suppose one can decompose the polytope $\Delta$ into a Minkowski sum of $r$ polytopes $\Delta_i$
$$
\Delta = \Delta_1 + \ldots + \Delta_r.
$$
Each of the polytopes $\Delta_i$ correspond to a torus-invariant nef divisor $D_i$. One then may take a regular section $f$ of the toric vector bundle $\bigoplus_i \mathcal{O}(D_i)$. Taking its zero locus gives a Calabi-Yau complete intersection $X_{(\Delta_i)}=Z(f)$ in the toric variety $\mathbb{P}_{\Delta}$. Borisov proposed a mirror Calabi-Yau to $X_{(\Delta_i)}$ \cite{Bor93}. Take the polytopes
$$
\nabla_i := \{ n \in N_{\R} | \langle \Delta_j, n\rangle \geq -\delta_{ij} \text{ for all $j$}\}.
$$
Take the Minkowski sum of these $\nabla_i$ to form $\nabla := \nabla_1 + \ldots + \nabla_r$. Each polytope $\nabla_i$ corresponds to a nef line bundle $E_i$ 
in the Fano toric variety $\mathbb{P}_{\nabla}$. Taking a regular section $g$ of the toric vector bundle $\bigoplus_i \mathcal{O}(E_i)$ and then taking its zero locus gives a Calabi-Yau complete intersection $X_{(\nabla_i)}=Z(g)$ in $\mathbb{P}_\nabla$.

Note that all of these varieties are highly singular. To remedy this, Batyrev gave a maximal projective crepant partial (MPCP) desingularization by replacing the toric varieties $\mathbb{P}_\Delta$ and $\mathbb{P}_\nabla$ with simplicial refinements of the normal fans $\Sigma_{N(\Delta)}$.  These are maximal projective crepant partial resolutions with at most orbifold singularities, hence, coarse moduli spaces of smooth DM stacks.  This resolves the singularities two ways; take the MPCP desingularization, then take the associated DM stack.

Let us denote by $Z_{(\Delta_i)}$ and $Z_{(\nabla_i)}$  the Calabi-Yau stack corresponding to the MPCP desingularization of the data of the nef partitions $\Delta_i$ and $\nabla_i$ and by $\P_{\Delta}$ and $\P_{\nabla}$ the MPCP desingularization  of the toric stacks associated to the polytopes $\Delta$ and $\nabla$.

In 1996, Batyrev and Borisov \cite{BB96} proved that: 

\begin{theorem}[Batyrev-Borisov]\label{BB96}
Mirror symmetry holds for $X_{(\Delta_i)}$, $X_{(\nabla_i)}$ on the level of stringy Hodge numbers:
$$
h^{p,q}_{\text{st}} (X_{(\Delta_i)}) = h^{d-r-p, q}_{\text{st}}(X_{(\nabla_i)}),
$$
i.e.,
$$
h^{p,q} (Z_{(\Delta_i)}) = h^{d-r-p, q}(Z_{(\nabla_i)}).
$$

\end{theorem}
Interestingly, there can sometimes be an ambiguity in the Batyrev-Borisov construction. Take $0$ to be the unique interior point in the reflexive polytope $\Delta$ and a nef-partition $\Delta_i$ of $\Delta$. Then suppose that there exists elements $p_i, p_i' \in \Delta_i$ so that
$$
0 = \sum_i p_i = \sum_i p_i'.
$$
Then there exists two nef partitions that are translates of one another:
$$
\Delta  = \sum_i \Delta_i = \sum_i (\Delta_i - p_i + p_i').
$$
Denote the translated polytope by  $\Delta_i' := \Delta_i - p_i + p_i'$. The generic Calabi-Yau complete intersections corresponding to $\Delta_i$ and $\Delta_i'$ are isomorphic
$$
X_{(\Delta_i)} \cong X_{(\Delta_i')}.
$$
Now take the dual polytopes $\nabla_i$ and $\nabla_i'$ associated to the nef partitions $\Delta_i$ and $\Delta_i'$. One finds that the polytope  $\nabla = \nabla_1 + \ldots+ \nabla_r$ is often different from $\nabla' = \nabla_1' + \ldots + \nabla_r'$. So, the Calabi-Yau complete intersection $X_{(\nabla_i)}$ corresponds to a zero locus of a section $g$ of the vector bundle $\bigoplus_i \mathcal{O}_{\mathbb{P}_\nabla}(E_i)$ while $X_{(\nabla_i')}$ corresponds to a zero locus of a section $g'$ of some other vector bundle $\bigoplus_i \mathcal{O}_{\mathbb{P}_{\nabla'}}(E_i')$. This means that the mirror Calabi-Yau complete intersections $X_{(\nabla_i)}$ and $X_{(\nabla_i')}$ are in possibly different toric varieties $\mathbb{P}_\nabla$ and $\mathbb{P}_{\nabla'}$. There is no natural isomorphism amongst the Batyrev-Borisov mirrors $X_{(\nabla_i)}$ and $X_{(\nabla_i')}$. In 2007, Batyrev and Nill formulated the following question about these varieties:
\begin{question}[Question 5.2 of \cite{BN07}]
Are the mirror Calabi-Yau complete intersections $X_{(\nabla_i)}$ and $X_{(\nabla_i')}$ birationally isomorphic?
\end{question}
This question was studied by Li and answered affirmatively (Theorem 4.6 of \cite{Li13}) using determinantal varieties.  Batyrev and Nill further conjectured the following:

\begin{conjecture}[Conjecture 5.3 of \cite{BN07}]
There exists an equivalence between the derived category of coherent sheaves on the Calabi-Yau complete intersections $Z_{(\nabla_i)}$ and $Z_{(\nabla_i')}$.
\label{conj: BN}
\end{conjecture}

In Section~\ref{sec: BN}, we supply a proof i.e. we prove the following Theorem.
\begin{theorem}[=Theorem \ref{BNconjecturetheorem}]\label{BNyes}
Conjecture~\ref{conj: BN} holds, meaning, there is a correspondence between complete intersections $Z_{(\nabla_i)}$ in $\P_\nabla$ and complete intersections $Z_{(\nabla_i')}$ in $\P_{\nabla'}$ such that
\[
\dbcoh{Z_{(\nabla_i)}} \cong \dbcoh{Z_{(\nabla_i')}}.
\]
\end{theorem}
Let $\gm$ denote the group of units of $\kappa$.  The main technique here is to compare the global quotient stacks
\[
Y := \left[\op{tot}\left(\bigoplus_i \mathcal{O}_{\P_\nabla}(-E_i) \right)/ \gm\right] \tand 
Y' = \left[\op{tot}\left(\bigoplus_i \mathcal{O}_{\P_{\nabla'}}(-E_i')\right) / \gm \right]
\]
 and establish them both as open substacks of $[\A^N /  S_\nu \times \gm]$ for a particular subgroup $S_\nu \subseteq \gm^N$.  Restricting $S_\nu$-invariant functions on $\A^N$ establishes a correspondence between functions on  $Y$ and $Y'$ whose critical loci are Calabi-Yau complete intersections in $\mathbb{P}_\nabla$ and $\mathbb{P}_{\nabla'}$. In our proof, we do not necessarily need genericity as in the birationality proof of \cite{Li13}, just that the sections $g$ and $g'$ give complete intersections.  Also, note that our proof does not prove anything about the birationality and uses different tools than that of Li.

\subsection{Plan of the Paper}
Here is a brief summary of how the paper is organized. 
In Section~\ref{sec: Background}, we give the necessary background on polytopes for the Batyrev-Borisov construction, define the multiple mirrors that we will prove are derived equivalent, and describe the Batyrev-Nill conjecture.
In Section~\ref{sec: categories}, we provide background on categories of singularities and in particular the theorems of Orlov, Isik, and Shipman which we will use.
 In Section~\ref{sec: algebraic}, we describe the recent literature on variations of GIT quotients for linear actions of abelian groups on affine space and its relationship with equivalences of categories of singularities.
 In Section~\ref{sec: toric interpretation}, we give a clearer context for Section~ \ref{sec: algebraic}  and provide details in the case of toric varieties. 
 In Section~\ref{sec: BN}, we look at nef partitions and Batyrev-Borisov mirrors, showing that the distinct mirrors that come from combinatorial ambiguity in the Batyrev-Borisov construction are derived equivalent, proving the Batyrev-Nill conjecture.

\vspace{2.5mm}
\noindent \textbf{Acknowledgments:}
We heartily thank Colin Diemer for suggesting that VGIT may relate to the double mirror picture and give special thanks to Charles Doran for input on this project from start to finish.
This project was also greatly aided by stimulating conversations and suggestions from many excellent mathematicians; Matthew Ballard, Ionut Ciocan-Fontanine, Ron Donagi, Bumsig Kim, Mark Gross, Daniel Halpern-Leistner, Umut Isik, Xenia de la Ossa, Helge Ruddat, and Vyacheslav Shokurov.  We also thank the referee for several useful comments and suggestions to improve the paper.

Furthermore, the first-named author is grateful to the Korean Institute for Advanced Study for their hospitality while this document was being prepared. The second-named author thanks the Pacific Institute for the Mathematical Sciences and NSERC for various forms of support.  Frequent visits to the University of Alberta greatly expedited the progress of this work.  

The first-named author is also indebted to the Natural Sciences and Engineering Research Council of Canada and Canada Research Chair Program for support provided by NSERC RGPIN 04596 and CRC TIER2 229953.
The second-named author acknowledges that this material is based upon work supported by the National Science Foundation under Award No.\ DMS-1401446 and the Engineering and Physical Sciences Research Council under Grant EP/N004922/1.

 \vspace{2.5mm}

\section{Preliminaries}
 \label{sec: Background}
\subsection{Lattices and Polytopes}

In this section, we provide a brief exposition on the mirror construction for Calabi-Yau complete intersections in toric Fano varieties by Batyrev and Borisov. We also provide a recap of the conjecture of Batyrev and Nill on the derived equivalence of toric ``double mirrors."  More detailed references for this construction include \cite{BB97, BN07, Li13}. 

Let $M$ be a lattice of rank $d$ and $N$ be its dual lattice with pairing 
$$
\langle , \rangle : M\times N \rightarrow \Z.
$$
Set $M_{\R}:= M\otimes \R$ and $N_{\R} = N \otimes \R$. We extend the pairing to $\langle , \rangle : M_{\R} \times N_{\R} \rightarrow \R$ in the natural way.  

\begin{definition}
A \newterm{polytope} is the convex hull of a finite set of vectors in $M_{\R}$.
\end{definition}

\begin{definition}
A \newterm{lattice polytope} is the convex hull in $M_{\R}$ of a finite set of vectors in $M$.
\end{definition}

\begin{definition}
We define the \newterm{dual polytope} $\Delta^\vee$ to a lattice polytope $\Delta$ in $M_{\R}$ to be 
$$
\Delta^\vee := \left\{ n \in N_{\R} \middle| \ \langle m, n\rangle \geq -1 \text{ for all } m \in\Delta\right\}\subseteq N_{\R}.
$$
\end{definition}

\begin{definition}
We say $\Delta$ is \newterm{reflexive} if the dual polytope $\Delta^\vee$ is a lattice polytope. If there exists an interior lattice point $m \in \Delta$ so that $(\Delta-m)^\vee$ is a lattice polytope, we say that $\Delta$ is \newterm{reflexive with respect to $m$}. We say the polytope $\Delta$ is \newterm{Gorenstein of index $r$} if $r\Delta$ contains an interior lattice point $m$ and the polytope $r\Delta -m$ is reflexive.
\end{definition}

We can extend the lattices $M$ and $N$ to lattices $\overline{M} = M \oplus \Z^r$ and $\overline{N} = N \oplus \Z^r$ and obtain a new pairing
$$
\langle, \rangle: \overline M \times \overline N \rightarrow \Z
$$
defined to be $\langle (m, a_i), (n, b_i) \rangle : = \langle m, n\rangle + \sum_i a_ib_i$.  Take $e_i$ to be the standard elementary basis for $\Z^r$ when added to $N$ to create the direct sum $\overline{N}$ and take its dual basis $e_i^*$ for $\Z^r$ when added to $M$ to create $\overline{M}$. Let $\sigma \subseteq \overline M _{\R}$ be a $(d+r)$-dimensional, strictly convex, polyhedral cone with vertex $0\in \overline M$. We define the dual cone $\sigma^\vee$ to the cone $\sigma$ to be
$$
\sigma^\vee : = \left\{ n \in \overline N_{\R} \middle| \ \langle m, n\rangle \geq 0 \text{ for all } m\in \sigma\right\}.
$$
The dual cone $\sigma^\vee$ is also a $(d+r)$-dimensional, strictly convex, rational, polyhedral cone with vertex $0\in \overline N$.

\begin{definition}
A $(d+r)$-dimensional rational polyhedral cone $\sigma$ is a \newterm{Gorenstein cone} if it is generated by finitely many lattice points that are contained in an affine hyperplane
$$
\left\{ m\in \overline M_{\R} \middle| \ \langle m, n\rangle = 1\right\}
$$
for some $n \in \overline N$. The point $n$ is uniquely determined. We call this element the degree element and we denote it by $\deg^\vee$.  We define the \newterm{$k^{\text{th}}$ slice} of $\sigma$ to be 
$$
\sigma_{(k)} :=\left\{ m \in \overline M_{\R} \middle| \ \langle m, \deg^\vee\rangle = k\right\}
$$
\end{definition}

Given a Gorenstein cone $\sigma$, the lattice polytope $\sigma_{(1)}$ is called the \newterm{support} of $\sigma$.

\begin{definition}
A Gorenstein cone $\sigma$ is called a \newterm{reflexive Gorenstein cone} if the dual cone $\sigma^\vee$ is also a Gorenstein cone. We denote the degree element relative to $\sigma^\vee$ to be $\deg \in M$.  We define the \newterm{index} of a reflexive Gorenstein cone to be the integer $\langle \deg, \deg^\vee\rangle$.
\end{definition}

We have the following proposition to relate many of these definitions.

\begin{proposition}[Proposition 2.11 of \cite{BB97}]
Let $\sigma$ be a Gorenstein cone $\sigma \subseteq M_{\R}$. The following are equivalent:
\begin{enumerate}[a)]
\item the cone $\sigma$ is a reflexive Gorenstein cone of index $r$;
\item the $r^{\text{th}}$ slice $\sigma_{(r)}$ of the cone $\sigma$ is a reflexive polytope; 
\item the polytope $\sigma_{(1)}$ is a Gorenstein polytope of index $r$.
\end{enumerate}
In this case, $\sigma_{(r)}$ is a reflexive polytope with unique interior point $\deg$.
\label{prop: BB}
\end{proposition}

\subsection{Nef Partitions, Toric double mirrors and the Batyrev-Nill Conjecture}

\begin{definition}
Let $\Delta \subseteq M_{\R}$ be a $d$-dimensional reflexive polytope with respect to $m$.  A \newterm{nef partition} of $\Delta$ of length $r$ is a decomposition of the polytope $\Delta$ into $r$ polytopes $\Delta_1, \ldots, \Delta_r \subseteq M_{\R}$ and a choice of lattice points $p_i \in \Delta_i$ so that the sum of the lattice points $p_i$ is $m$ and the Minkowski sum $\Delta_1 + \ldots + \Delta_r$ is equal to the polytope $\Delta$.
\end{definition}

The \newterm{Cayley polytope} of length $r$ is 
$$
\Delta_1 * \ldots * \Delta_r := \operatorname{Conv}((\Delta_1+ e_1^*), \ldots, (\Delta_r + e_r^*)) \subseteq \overline M_{\R}
$$
The cone
$$
\R_{\geq 0} (\Delta_1 * \ldots * \Delta_r)  = \R_{\geq 0} (\Delta_1 +e_1^*)+ \ldots + \R_{\geq 0} (\Delta_r + e_r^*) 
$$
is the \newterm{Cayley cone} associated to the nef partition $\{\Delta_i\}$.

If the affine space generated by the polytopes $\Delta_i$ is the vector space $M_{\R}$, then the Cayley cone is a Gorenstein cone of dimension $d+r$ with support $ \Delta_1 * \ldots * \Delta_r$.

\begin{proposition}[Proposition 2.3 of \cite{BN07}]
Let $\sigma \subseteq \overline M_{\R}$ be a Gorenstein cone of dimension $d+r$. Then the following statements are equivalent:
\begin{enumerate}[a)]
\item the cone $\sigma$ is a Cayley cone associated to $r$ lattice polytopes;
\item the support $\sigma_{(1)}$ is a Cayley polytope of length $r$;
\item There is a lattice projection $\overline M \rightarrow \Z^r$, which maps $\sigma_{(1)}$ surjectively to the polytope $\operatorname{Conv}(e_1^*, \ldots, e_r^*)$; 
\item There are nonzero elements $e_1', \ldots, e_r' \in \sigma^\vee \cap \overline N$ such that $\sum_i e_i' = \deg^\vee$.
\end{enumerate}
Moreover, the lattice vectors $e_i^*$ form a part of a basis of $\overline N$ and the Cayley structure of $\sigma_{(1)}$ is uniquely determined by $t$ polytopes 
$$
\Delta_i : = \left\{ x \in \sigma_{(1)} \middle| \ \langle x, e_j'\rangle = 0 \text{ for $j \neq i$}\right\} - e_i^*
$$
for each $i$ where $e_i$ are part of a dual basis. These polytopes have the property that $\langle \Delta_i, e_i'\rangle = 1$.
\label{prop: Cayley comparison}
\end{proposition}

We can now define a class of Cayley cones that is maximal in some sense:

\begin{definition}[Definition 2.4 of \cite{BN07}]
A reflexive Gorenstein cone $\sigma$ is called \newterm{completely split} if it can be realized as a Cayley cone associated to $r$ lattice polytopes where $r$ is the index of the cone, i.e., $\langle \deg, \deg^\vee\rangle = r$.
\end{definition}

We can tell which polytopes $\Delta$ yield a nef partition. Consider the following definition:

\begin{definition}[Definition 2.7 of \cite{BN07}]
Let $\Delta \subseteq M_{\R}$ be a $d$-dimensional lattice polytope. A simplex $S$ spanned by $r$ affinely independent lattice points in $\Delta$ is called a \newterm{special $(r-1)$-simplex} of $\Delta$ if each facet of $\Delta$ contains exactly $r-1$ vertices of $S$.
\end{definition}
We now have the following equivalent conditions 

\begin{proposition}[Proposition 3.6 of \cite{BN07}]
 Let $\Delta$ be a polytope which is reflexive with respect to $m$ and $\Delta_1, \ldots, \Delta_r$ be lattice polytopes such that 
 \[
 \Delta_1 + \ldots + \Delta_r = \Delta.
 \]
The following are equivalent:
\begin{enumerate}[a)]
\item The dual of the Cayley cone is a completely split reflexive Gorenstein cone of index $r$;
\item The Cayley polytope $\Delta_1 *\cdots * \Delta_r$ is a Gorenstein polytope of index $r$ containing a special $(r-1)$-simplex; 
\item 
There exists lattice points $p_i \in \Delta_i$ such that $\sum p_i = m$, i.e., $ \Delta_1 + \ldots + \Delta_r = \Delta$ is a nef-partition.
\end{enumerate}
\label{prop: BN nef}
\end{proposition}

Let
\[
p_i \in \Delta_i, \Delta_1 + \ldots + \Delta_r = \Delta
\]
be a nef partition.  For any $j$, we define the polytope 
$$
\nabla_j = \left\{ n \in N_{\R} \middle| \langle m, n \rangle \geq - \delta_{ij} \text{ for all $m \in \Delta_i - p_i$ for all $i$}\right\}.
$$
It can be shown that for all $i$, the polytope $\nabla_i$ is a lattice polytope. We can define a polytope $\nabla$ that is the Minkowski sum of all the $\nabla_i$:
$$
\nabla := \nabla_1+\ldots +\nabla_r.
$$

\begin{proposition}[Propositions 3.13 and 3.18 of \cite{BN07}]
The $r$ polytopes $\nabla_j$ define a nef-partition. The polytope $\op{Conv}(\Delta_1, \ldots, \Delta_r)$ is dual to $\nabla$. Analogously, the polytope $ \op{Conv}(\nabla_1, \ldots, \nabla_r) $ is dual to $\Delta - m$.
\end{proposition}

We now consider the case where $\Delta$ contains two different special $(r-1)$-simplices $S$ and $S'$ and the dual polytope $\Delta^\vee$ contains a special $(r-1)$-simplex $T$.  The two special simplices $S$ and $S'$ define two collections of lattice points $p_i$ and $p_i'$ in $\Delta_i$ such that 
$$
0 = \sum_i p_i = \sum_i p_i',
$$
 by Proposition~\ref{prop: BN nef}. By taking the dual nef partitions $\nabla_i$ to $\Delta_i$ and $\nabla_i'$ to $\Delta_i' := \Delta_i -p_i + p_i'$, one may get very different polytopes. The corresponding generic complete intersection $Z_{\Delta_i}$ in $X_\Delta$ is isomorphic to some corresponding generic complete intersection $Z_{\Delta_i'}$ in $X_\Delta$. When it comes to the dual nef partitions and their generic complete intersections $Z_{(\nabla_i)}$ and $Z_{(\nabla_i')}$, we are not given such a nice correspondence. 
  
\begin{example}
We repeat Example 5.1 of \cite{BN07} for the convenience of the reader.  The relevant polytopes are pictured in Figure~\ref{fig: BN}. 

  \begin{center}
  \begin{figure}[h]
 \begin{tikzpicture}
   [scale=.5, vertex/.style={circle,draw=black!100,fill=black!100,thick, inner sep=0.5pt,minimum size=0.5mm}, cone/.style={->,very thick,>=stealth}]
   \filldraw[fill=black!20!white,draw=white!100]
     (-2.2,2.2) -- (2.2,2.2) -- (2.2,-2.2) -- (-2.2,-2.2) -- (-2.2,2.2);
  \draw[thick](0,0) -- (1,1);
    \draw[thick](0,0) -- (-1,1);
        \draw[thick](1,1) -- (-1,1);
          \draw[thick](0,0) -- (0,-1);
  \node at (1.4,1) {$\Delta_1$};
    \node at (1.4,-1) {$\Delta_2$};
  \foreach \x in {-2,-1,...,2}
   \foreach \y in {-2,-1,...,2}
   {
     \node[vertex] at (\x,\y) {};
   }
 \end{tikzpicture}
 \begin{tikzpicture}
   [scale=.5, vertex/.style={circle,draw=black!100,fill=black!100,thick, inner sep=0.5pt,minimum size=0.5mm}, cone/.style={->,very thick,>=stealth}]
   \filldraw[fill=black!20!white,draw=white!100]
     (-2.2,2.2) -- (2.2,2.2) -- (2.2,-2.2) -- (-2.2,-2.2) -- (-2.2,2.2);
          \draw[thick](0,-1) -- (1,0);
    \draw[thick](0,-1) -- (-1,0);
        \draw[thick](1,0) -- (-1,0);
                \draw[thick](0,0) -- (0,1);
  \node at (1.4,1) {$\Delta_1'$};
    \node at (1.4,-1) {$\Delta_2'$};
  \foreach \x in {-2,-1,...,2}
   \foreach \y in {-2,-1,...,2}
   {
     \node[vertex] at (\x,\y) {};
   }
 \end{tikzpicture}
 \begin{tikzpicture}
   [scale=.5, vertex/.style={circle,draw=black!100,fill=black!100,thick, inner sep=0.5pt,minimum size=0.5mm}, cone/.style={->,very thick,>=stealth}]
   \filldraw[fill=black!20!white,draw=white!100]
     (-2.2,2.2) -- (2.2,2.2) -- (2.2,-2.2) -- (-2.2,-2.2) -- (-2.2,2.2);
  \draw[thick](0,0) -- (1,1);
    \draw[thick](0,0) -- (-1,1);
        \draw[thick](1,1) -- (-1,1);
          \draw[thick](0,-1) -- (1,0);
    \draw[thick](0,-1) -- (-1,0);
        \draw[thick](1,0) -- (-1,0);
  \node at (1.4,1) {$\nabla_2$};
    \node at (1.4,-1) {$\nabla_1$};
  \foreach \x in {-2,-1,...,2}
   \foreach \y in {-2,-1,...,2}
   {
     \node[vertex] at (\x,\y) {};
   }
 \end{tikzpicture}
 \begin{tikzpicture}
   [scale=.5, vertex/.style={circle,draw=black!100,fill=black!100,thick, inner sep=0.5pt,minimum size=0.5mm}, cone/.style={->,very thick,>=stealth}]
   \filldraw[fill=black!20!white,draw=white!100]
     (-2.2,2.2) -- (2.2,2.2) -- (2.2,-2.2) -- (-2.2,-2.2) -- (-2.2,2.2);
  \draw[thick](1,0) -- (1,1);
    \draw[thick](-1,0) -- (1,0);
        \draw[thick](1,1) -- (-1,1);
        \draw[thick](-1,1) -- (-1,0);
          \draw[thick](0,0) -- (0,-1);
            \node at (1.4,1) {$\nabla_2'$};
    \node at (1.4,-1) {$\nabla_1'$};
  \foreach \x in {-2,-1,...,2}
   \foreach \y in {-2,-1,...,2}
   {
     \node[vertex] at (\x,\y) {};
   }
 \end{tikzpicture}
  \begin{tikzpicture}
   [scale=.5, vertex/.style={circle,draw=black!100,fill=black!100,thick, inner sep=0.5pt,minimum size=0.5mm}, cone/.style={->,very thick,>=stealth}]
   \filldraw[fill=black!20!white,draw=white!100]
     (-2.2,2.2) -- (2.2,2.2) -- (2.2,-2.2) -- (-2.2,-2.2) -- (-2.2,2.2);
  \draw[thick](0,-1) -- (2,1);
    \draw[thick](2,1) -- (-2,1);
        \draw[thick](-2,1) -- (0,-1);
    \node at (1.6,0) {$\nabla$};
  \foreach \x in {-2,-1,...,2}
   \foreach \y in {-2,-1,...,2}
   {
     \node[vertex] at (\x,\y) {};
   }
 \end{tikzpicture}
  \begin{tikzpicture}
   [scale=.5, vertex/.style={circle,draw=black!100,fill=black!100,thick, inner sep=0.5pt,minimum size=0.5mm}, cone/.style={->,very thick,>=stealth}]
   \filldraw[fill=black!20!white,draw=white!100]
     (-2.2,2.2) -- (2.2,2.2) -- (2.2,-2.2) -- (-2.2,-2.2) -- (-2.2,2.2);
  \draw[thick](-1,-1) -- (1,-1);
    \draw[thick](1,-1) -- (1,1);
        \draw[thick](1,1) -- (-1,1);
                \draw[thick](-1,1) -- (-1,-1);
    \node at (1.6,0) {$\nabla'$};
  \foreach \x in {-2,-1,...,2}
   \foreach \y in {-2,-1,...,2}
   {
     \node[vertex] at (\x,\y) {};
   }
 \end{tikzpicture}
 \caption{The polytopes $\Delta_1, \Delta_2, \Delta_1', \Delta_2', \nabla_1, \nabla_2, \nabla_1', \nabla_2', \nabla$, and $\nabla'$.} \label{fig: BN}
 \end{figure}
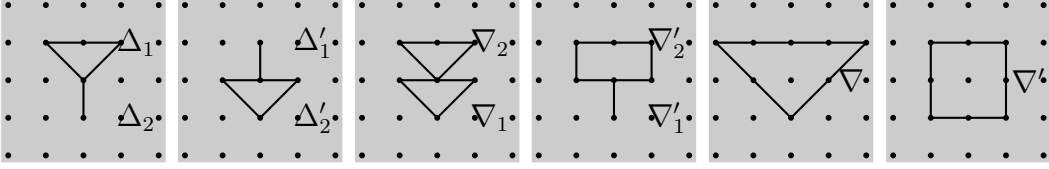
  \end{center}

We have, 
\[
X_{\nabla} = \mathbb P(1:1:2) \tand X_{\nabla'}  = \P^1 \times \P^1. 
\]
We label coordinates on $\mathbb P(1:1:2)$ by $a,b,c$ where $c$ has degree $2$ and coordinates on $\P^1 \times \P^1$ by $x_0,x_1, y_0, y_1$.  
The dual nef-partitions $\nabla_1, \nabla_2$ and $\nabla_1', \nabla_2'$ correspond to the vector bundles  $\O(2) \oplus \O(2)$ and $\O(2,1) \oplus \O(0,1)$.
The sections of these bundles are identified via the isomorphism
\[
\nabla_1 * \nabla_2 \cong \nabla_1' * \nabla_2'.
\]
We list the monomial section pairing in coordinates for this isomorphism in Figure~\ref{fig: ex mon list} where we label sections of the second summand of $\O(2)$ with a prime.
By inspection, the multiple mirrors are isomorphic in this case and there are two cases; single points and double points.  

\begin{center}
\begin{figure}[h]
\begin{tabular}{| c | c | }
\hline
$c$ & $y_0$ \\  \hline
$c'$ & $y_1$ \\ \hline
$a^2$ & $x_0^2y_0$ \\ \hline
$ab$ & $x_0x_1y_0$ \\ \hline
$b^2$ & $x_1^2y_0$ \\ \hline
$(a')^2$ & $x_0^2y_1$ \\ \hline
$a'b'$ & $x_0x_1y_1$ \\ \hline
$(b')^2$ & $x_1^2y_1$ \\ \hline
\end{tabular}
\caption{Monomials associated to lattice points in $\nabla_1 * \nabla_2 \cong \nabla_1' * \nabla_2'$. } \label{fig: ex mon list}
\end{figure}
\end{center}

\end{example}

Recently, Z. Li \cite{Li13} proved that the ambiguity of dual nef-partitions is rectified by birationality.
  \begin{theorem}[Theorem 4.6 of \cite{Li13}]
The generic complete intersection $Z_{(\nabla_i)}$ is birational to a complete intersection $Z_{(\nabla_i')}$.
 \end{theorem}
 
 This answers affirmatively Question 5.2 of \cite{BN07}.  Batyrev and Nill further conjectured (Conjecture 5.3 of \cite{BN07}) that the derived categories of coherent sheaves on $Z_{(\nabla_i)}$ and $Z_{(\nabla_i')}$ are equivalent. Theorem~\ref{BNconjecturetheorem} proves this conjecture at the level of stacks.

\section{Categories of Singularities}
\label{sec: categories}
In this section, we discuss categories of singularities, which first appeared in \cite{Orl04}.
 For our purposes, it will be necessary to define them for a certain class of stacks, and recollect some basic aspects about them in this context.

Let $X$ be a variety and $G$ be an algebraic group acting on $X$.

\begin{definition}
An object of $\dbcoh{[X/G]}$ is called \newterm{perfect} if it is locally quasi-isomorphic to a bounded complex of vector bundles.  We denote the full subcategory of perfect objects by $\op{Perf}([X/G])$.  The Verdier quotient of $\dbcoh{[X/G]}$ by $\op{Perf}([X/G])$ is called the \newterm{category of singularities} and denoted by
\[
\dsing{[X/G]} := \dbcoh{[X/G]} / \op{Perf}([X/G]).
\]
\end{definition}

 Consider a variety $X$ with the action of an algebraic group $G$ and a $G$-equivariant vector bundle $\mathcal{E}$ on $X$.  Take the $G$-invariant section $s \in \op{H}^0(X, \mathcal E)$ and consider the zero locus $Z$ of $s$ in $X$.  The pairing with $s$ induces a global function on the total space of $\mathcal E^\vee$.  We denote by $Y$ the zero locus of the pairing with $s$.  We now consider the fiberwise dilation action of $\gm$ on $Y$, which helps us relate the category of singularities with the bounded derived category.
 \begin{theorem}[Isik, Shipman]
 Suppose the Koszul complex on $s$ is exact.  Then there is an equivalence of categories
 \[
 \dsing{[Y/(G \times \gm)]} \cong \dbcoh {[Z/G]}.
 \]
 \label{thm: isik-shipman}
 \end{theorem}
 
 \begin{proof}
 The statement without $G$ is exactly that in \cite{Isik}.  Using the equivalent category of factorizations which we avoid in this paper for technical simplicity, the result is also in \cite{Shipman}. 
 
 The proof in \cite{Isik} extends to $G$-equivariant sheaves, as the main technical tool used in \cite{MR11} extends to $G$-equivariant sheaves (see Section 4.3 of \cite{MR11}).

One can also extend the proof in \cite{Shipman} using the category of $G$-equivariant factorizations \cite{BFK11}.  Here, the equivalence is given by the $G$-equivariant pullback to the fiber of $\op{tot}(\mathcal E^\vee)$ over $[Z/G]$  followed by the $G$-equivariant pushforward to $\op{tot}(\mathcal E^\vee)$.
\end{proof}

\section{Torus Actions on Affine Space}
\label{sec: algebraic}

Consider an affine space $X:= \mathbb A^{n+r}$ with coordinates $x_i, u_j$ for $1 \leq i \leq n, 1 \leq j \leq r$.
Let $T = \gm^{n+r}$ be the open dense torus with the standard embedding and action on $X$.
Take $S \subseteq T$ to be a subgroup and let $\widetilde{S}$ be the connected component of the identity.

We now consider the possible GIT quotients for the action of $\widetilde{S}$ on $X$ \cite{MFK}.  To do so, we need an ample equivariant line bundle on $X$.  By necessity, we take the trivial bundle with some equivariant structure.  The choice of this equivariant structure amounts to a character of $\widetilde{S}$, i.e., an element $\chi \in \op{Hom}(\widetilde{S}, \gm)$ from which we obtain a line bundle $\O_\chi$ by pulling back the representation of $\widetilde{S}$ via the morphism of stacks $[X/\widetilde{S}] \to [\op{pt}/\widetilde{S}]$.  

One can create a nice pictoral understanding this choice.  First, extend the choice of $\chi$ by letting it be an arbitrary element of $\op{Hom}(\widetilde{S}, \gm) {\otimes_{\Z}} \Q$.  One can get an equivariant line bundle simply by rationalizing denominators.  By looking at the associate line bundle, each $\chi$ determines an open subset $U_\chi$ corresponding to the semi-stable locus of $X$ with respect to $\chi$.  Since this semi-stable locus only depends of the ray which $\chi$ lies on, no harm is done by introducing rational coefficients.  There is a fan $\Sigma_{\op{GKZ}}$ in $\op{Hom}(\widetilde{S}, \gm) {\otimes_{\Z}} {\Q}$ called the \newterm{GKZ-fan} (or \newterm{secondary fan}), where $U_\chi$ is constant on the interior of each cone in the fan.  The maximal cones of this fan are called \newterm{chambers} and the codimension 1 cones are called \newterm{walls}.  

Furthermore, given the action of $\widetilde{S}$ on $X$, there are finitely many chambers $\sigma_1, ..., \sigma_t$ in the fan $\Sigma_{\op{GKZ}}$.  To declutter notation, given any character $\chi_p$ in the interior of $\sigma_p$ we can consider the semi-stable points with respect to that character.  This yields an open subset in $X$ which we denote by $U_p$. Excellent and far more extensive references for this discussion of toric GIT are \cite{GKZ, CLS}.

\begin{definition}
Let $\times : \gm^{n+r} \to \gm$ be the multiplication map.  We say that $S$ satisfies the \newterm{quasi-Calabi-Yau condition} if $\times|_{\widetilde{S}} = \bf{1}$, i.e., the multiplication map restricted to $\widetilde{S}$ is the trivial homomorphism.
\end{definition}

\begin{definition}
 Let $G$ be a group acting on a space $X$ and let $f$ be a global function on $X$.  We say that $f$ is \newterm{semi-invariant} with respect to a character $\chi$ if, for any $g \in G$,
 \[
 f(g \cdot x) = \chi(g)f(x).
 \]
 Equivalently, this means that $f$ is a section of the equivariant line bundle $\O(\chi)$ on the global quotient stack $[X/G]$.
\end{definition}

\begin{remark}
Each variable $x_i$ is semi-invariant with respect to a unique character of $S$ which we can denote by $\op{deg}(x_i)$.   The quasi-Calabi-Yau condition is equivalent to
\begin{equation}
\sum \op{deg}(x_i)+ \sum \op{deg}(u_j)
\end{equation}
being torsion.  
\end{remark}

As an aside, the derived categories $\dbcoh{[U_p / S]}$ are all equivalent \cite{Kaw05}.  However, to suit our needs, namely to use Theorem \ref{thm: isik-shipman}, we will add an auxiliary $\gm$-action and an $S$-invariant function which is semi-invariant with respect to the projection character $\chi: S \times \gm \to \gm$.  
This auxiliary $\gm$-action  acts with weight $0$ on the $x_i$ for all $i$ and with weight $1$ on the $u_j$  for  all $j$.  Using a term from the string theory literature, this auxiliary action will be referred to as \newterm{$R$-charge}.

Also, notice the action of $S$ on $\op{Spec} \kappa[u_j]$ yields a character $\gamma_j$ of $S$. 
For each $i=1, \ldots, r$, let $f_i$ be a function in $\kappa[x_j]$ that is $S$-semi-invariant with respect to the character $\gamma_j^{-1}$. These functions $f_i$ determine a complete intersection in $\A^n$
 by defining the function  
\[
w := \sum_{j=1}^r u_j f_j.
\]
Notice that $w$ is $S$-invariant, i.e., homogeneous of degree 0 and homogeneous of degree $1$ with respect to the $R$-charge.  From a geometric perspective, this means $w$ is a section of the equivariant bundle $\O(\chi)$ where $\chi: S \times \gm \rightarrow \gm$ is the projection character, or equivalently, as mentioned above, $w$ is semi-invariant with respect to $\chi$.   We call $w$ the \newterm{superpotential}.
Let $Z$ denote the zero-locus of $w$ in $X$ and 
\[
Z_p := Z \cap U_p.
\]

\begin{theorem}[Herbst-Walcher]
If $S$ satisfies the quasi-Calabi-Yau condition, there is an equivalence of categories,
\[
\dsing{[Z_p/ S \times \gm]} \cong \dsing{[Z_q/ S \times \gm]}
\]
for all $1 \leq p, q \leq r$.
\label{thm: HW}
\end{theorem}
\begin{proof}
This is essentially Theorem 3 of \cite{HW} stated in geometric as opposed to algebraic language.  For the geometric translation see \cite{HL15} Cor 4.8 and Prop 5.5.  One can also use Theorem 5.2.1 of \cite{BFK12}, this works for the equivalent language of factorization categories as in Theorem 3.5.2 of loc.\ cit.\  

Technically, these theorems only state the derived equivalence for characters in adjacent chambers of the GKZ fan (the fact that wall-crossing across two adjacent chambers in the GKZ fan is elementary is Proposition 5.1.4 of loc.\ cit.) However, the equivalences can be combined to get the statement for all chambers since the GKZ fan, is a fan with convex support by Theorem 14.4.7 of \cite{CLS}.
\end{proof}

\section{Toric Interpretation}
\label{sec: toric interpretation}

In this section, we provide concrete combinatorial descriptions of the results in Section~\ref{sec: algebraic} in the language of toric geometry.  More precisely, we review the fan description of toric vector bundles. This combinatorial description is used in Section~\ref{sec: BN} to outline the Batyrev-Borisov mirror construction \cite{Bat, Bor93, BB94, BB97}. We will in this chapter give a stacky treatment of these toric vector bundles, reviewing of various results of Chapters 14 and 15 of \cite{CLS} as needed. 

\subsection{The GKZ, Triangulations, and Torus Actions}

Following \S 15.2 of \cite{CLS}, take $\nu= (v_1, \ldots, v_t)$ to be a collection of nonzero points in $N$ lying in an affine hyperplane. Denote by $Q_\nu$ the polytope that is the convex hull of the finite set $\nu$. A \newterm{triangulation} $\mathcal{T}$ of $\nu$ is a collection of simplices so that:
\begin{enumerate}[i.]
\item Each simplex in $\mathcal{T}$ has codimension 1 in $N_{\R}$ with all of its vertices in $\nu$.
\item The intersection of any two simplices in $\mathcal{T}$ are a face of both simplices.
\item The union of the simplices in $\mathcal{T}$ is the convex hull $Q_{\nu}$.
\end{enumerate}

There is a correspondence between the set of all triangulations of $\nu$ and the set of all simplicial fans whose support is $\operatorname{Cone}(\nu)$ with each of its 1-rays being generated by $v_i$ for some $i$. 

We want to restrict our attention to a particular subset of triangulations. Given nonnegative weights $\omega = (w_1, \ldots, w_t)$, we get the cone
$$
C_{\mathbf{\nu,\omega}} = \operatorname{Cone}((v_1,w_1), \ldots, (v_t, w_t)) \subseteq N_{\R} \times \R.
$$

The lower hull of $C_{\mathbf{\nu,\omega}}$ is the collection of all cones that that are facets of $C_{\mathbf{\nu,\omega}}$ whose inward pointing normal has a positive last coordinate.

Taking the cones in the lower hull and applying the projection $N_{\R} \times \R \rightarrow N_{\R}$ gives a collection of cones in $N_{\R}$. Let $\Sigma_{\omega}$ be the fan in $N_{\R}$ consisting of all these cones and their proper faces. The fan $\Sigma_{\omega}$ has support $|\Sigma_{\omega}|=C_\nu$ and $\Sigma_{\omega}(1) \subseteq \left\{ \operatorname{Cone}(v_i) \middle| \ v_i \in \nu\right\}$.

\begin{definition}
A triangulation $\mathcal{T}$ of $\nu$ is \newterm{regular} if there are weights $\omega$ such that $\Sigma_\omega$ is simplicial and $\mathcal{T} = \Sigma_\omega \cap Q_\nu$.
\end{definition}

\begin{definition}
A toric variety $X_\Sigma$ is \newterm{semiprojective} if it is projective over an affine and has a torus fixed point. 
\end{definition}

We can move between triangulations and actions on affine space as follows.
Consider $ \nu = (v_1, ..., v_n) \subseteq N$ where $N$ is a lattice of dimension $d$. 
 We get a right exact sequence
\begin{align}
\label{eq: f}
M & \overset{f_\nu}{\longrightarrow} \Z^n \overset{\pi}{\longrightarrow} \op{coker}(f_\nu) \to 0 \\
m & \mapsto \sum_{i=1}^t \langle v_i, m \rangle e_i.  \notag
\end{align}
Applying $\op{Hom}(-, \gm)$, we get a left exact sequence
\[
0 \longrightarrow \op{Hom}(\op{coker}(f_\nu), \gm) \overset{\widehat{\pi}}{\longrightarrow} \gm^{n} \overset{\widehat{f_{\nu}}}{\longrightarrow} \gm^d
\]
We set 
\[
S_\nu := \op{Hom}(\op{coker}(f_\nu), \gm)
\]
The inclusion $S_\nu \to \gm^{n}$ gives the setup of Section~\ref{sec: algebraic} with $S_\nu$ acting on $\A^{n}$ via this inclusion.

Similarly, starting with a subgroup $S$ of $\gm^{n}$
\[
0 \longrightarrow S  \overset{i_S}{\longrightarrow} \gm^{n} \overset{p}{\longrightarrow} \op{Coker }(i_S) \to 0
\]
We may apply $\op{Hom}(-, \gm)$ to get
\[
 \op{Hom}(\op{Coker }(i_S), \gm)  \overset{\widehat{p}}{\longrightarrow} \Z^{n} \overset{\widehat{i_S}}{\longrightarrow} \op{Hom}(S, \gm) \to 0
\]
So we may set $\nu_i(S)$ to be the element of  $\op{Hom}(\op{Coker }(i_S), \gm)^{\vee}$ given by the composition of $\widehat{i_S}$ with the projection of $\Z^n$ onto its $i^{\op{th}}$ factor and define
\[
\nu(S) := (\nu_1(S), ...., \nu_n(S)).
\]

\begin{proposition}
We have $S_{\nu(S)} = S$.  Furthermore, if $f_\nu$ is injective then $\nu(S_\nu) =\nu$ .
\end{proposition}
\begin{proof}
Starting with $S$ we set $N =  \op{Hom}(\op{Coker }(i_S), \gm)^\vee$ so that $M =  \op{Hom}(\op{Coker }(i_S), \gm)_{\op{free}}$. The map $f_{\nu(S)}$ is nothing more than the map $\widehat{p}$.  Hence, the first statement follows from the fact that $S = \op{Hom}(\op{Hom}(S, \gm), \gm)$.

Starting with $\nu$, we have $S = \op{Hom}(\op{coker}(f_\nu), \gm)$ and  $i_S = \widehat{\pi}$.  Since $f_\nu$ is injective,  $\widehat{f_\nu}$ is surjective and hence $\gm^d = \op{Coker }i_{S_\nu}$.  Therefore, $\widehat{p} = f_\nu$ and 
\begin{align*}
\nu_i(S) & = \widehat{p}^\vee(e_i) \\
&  = f_\nu^\vee(e_i) \\
& = \nu_i.
\end{align*} \end{proof}

Finally, consider a fan $\Sigma \subseteq N_{\R}$ with $n$ rays.
We can associate a new fan
\[
\op{Cox}(\Sigma) := \{ \op{Cone}(e_\rho | \ \rho \in \sigma) | \ \sigma \in \Sigma \} \subseteq \R^n = \R^{\Sigma(1)}.
\]
Enumerating the rays, this fan is a subfan of the standard fan for $\A^n$:
\[
\Sigma_n := \{ \op{Cone} (e_i | \ i \in I) | \ I \subseteq \{1, ..., n \} \}.
\]
 Hence $X_{\op{Cox}(\Sigma)}$ is an open subset of $\A^{n}$.

\begin{definition}
We call $X_{\op{Cox}(\Sigma)}$ the \newterm{Cox open set} associated to $\Sigma$.  We define the \newterm{Cox stack} associated to $\Sigma$ to be
\[
\mathcal{X}_{\Sigma}:= [X_{\op{Cox}(\Sigma)}/S_{\Sigma(1)} ].
\]
If $\Delta \subseteq M$ is a polyhedron then we take $\mathcal{X}_{\Delta} = \mathcal{X}_{\Sigma}$ where $\Sigma$ is the normal fan to $\Delta$.
\end{definition}

\begin{theorem}
If $\Sigma$ is simplicial, then the coarse moduli space of $\mathcal{X}_{\Sigma}$ is $X_\Sigma$.  When $\Sigma$ is smooth (or equivalently $X_\Sigma$ is smooth) $\mathcal{X}_{\Sigma} \cong X_\Sigma$.
\label{thm: stack realization}
\end{theorem}
\begin{proof}
This a combination of Proposition 5.1.9 and Theorem 5.1.11 in \cite{CLS}.
\end{proof}

If $\mathcal T$ is a triangulation of $\nu$ then we can form a similar stack.  We will see below that this stack is isomorphic as long as $\nu$ spans the lattice $N$.
\begin{definition}
Let $\mathcal T$ be a triangulation of $\nu$.  We define the \newterm{Cox stack} associated to $\mathcal T$ to be
\[
\mathcal X_{\mathcal T} := [X_{\op{Cox}(\Sigma)} \times \gm^{|\nu \backslash \Sigma(1)|} /S_\nu ]
\]
where $\Sigma$ is the fan associated to $\mathcal T$.
\end{definition}

\begin{proposition}
Let $\mu \subseteq \nu$ be finite sets in $N$ with the same convex hull.  Let $\mathcal T_\mu$ be a triangulation of $\mu$ and $\mathcal T_\nu$ be the same triangulation regarded as a triangulation of $\nu$. 
Assume that $\mu$ spans $N_{\R}$.  Then, there is an isomorphism of stacks
\[
\mathcal X_{\mathcal T_\mu} \cong \mathcal X_{\mathcal T_\nu}.
\]
\label{prop: same triangulations}
\end{proposition}
\begin{proof}
Since $\mu$ spans $N_{\R}$, the map $f_\mu$ is injective.  Thus, we have the following diagram determined by the snake lemma:
\[
\begin{CD}
0 @>>> 0 @>>> \Z^{|\nu \setminus \mu|} @= \Z^{|\nu \setminus \mu|}  @>>> 0\\
@. @VVV @VVV @VVV @.\\
0 @>>> M @>f_\nu>> \Z^{|\nu|} @>>> \op{coker}(f_\nu) @>>> 0\\
@. @| @VVV @VVV @.\\
0 @>>> M @>f_\mu>> \Z^{|\mu|} @>>> \op{coker}(f_\mu) @>>> 0.
\end{CD}
\]
By definition, applying the exact functor $\op{Hom}( - , \gm)$ induces an action of $S_\nu$ on $\A^{|\nu|}$ and $S_\mu$ on $\A^{|\mu|}$.   

We claim there is an isomorphism of stacks
\[
[\A^{|\mu|} \times \gm^{|\nu \setminus \mu|} / S_\nu] \cong [\A^{|\mu|} / S_\mu]
\]
which restricts to the desired isomorphism 
\[
\mathcal X_{\mathcal T_\mu} \cong \mathcal X_{\mathcal T_\nu}
\]
simply by localizing.

This isomorphism comes from reducing $[\A^{|\mu|} \times \gm^{|\nu \setminus \mu|} / S_\nu] $ by  a free action of $\gm^{|\nu \setminus \mu|}$.  We can, dually, characterize this as a reduction of the $\op{coker}(f_\nu)$-graded algebra,
\[
\kappa[x_v | \ v \in \mu ] \otimes_{\kappa} \kappa[x_w, x_w^{-1} | \ w \in \nu \setminus \mu],
\]
to the $\op{coker}(f_\mu)$-graded algebra,
\[
 \kappa[x_v | \ v \in \mu ].
\]
Indeed, since the $x_w$ are units, this reduction comes from setting $x_w =1$ and quotienting $\op{coker}(f_\nu)$ by the subgroup generated by the degrees of $x_w$ for all $w \in \nu \setminus \mu$.  From the diagram above, this quotient is $\op{coker}(f_\mu)$.  So we obtain the desired isomorphism of graded algebras and the result follows.

\end{proof}

\begin{corollary}
Let $\Sigma$ be the fan associated to  a triangulation $\mathcal T_\mu$ of $\mu$ such that $\mu = \Sigma(1)$. Let $\nu$ be a finite set so that the convex hull of $\nu$ is equal to that of the convex hull of $\mu$ and let $\mathcal{T}_\nu$ be the same triangulation as above but as a triangulation of $\nu$ instead of $\mu$.  Assume that $X_\Sigma$ has no torus factors or, equivalently, $\mu$ spans $N_{\R}$.  Then, there is an isomorphism of stacks
\[
\mathcal X_{\Sigma} \cong \mathcal X_{\mathcal T_\nu}.
\]
\label{cor: reduce triangulation}
\end{corollary}
\begin{proof}
The fact that  $\mu$ spans $N_{\R}$ is equivalent to $X_\Sigma$ having no torus factors is Proposition 3.3.9 of \cite{CLS}.

The result is then the special case of Proposition~\ref{prop: same triangulations} where $\mu =  \Sigma(1)$.
\end{proof}

\begin{theorem}
Let $\nu$ be a finite subset of $N$ which spans $N_{\R}$.  There is a bijection between chambers of the GKZ fan for the action of $S_\nu$ on $\A^{|\nu|}$ and regular triangulations of $\nu$ which takes a chamber $\sigma_p = (\Sigma_p, \emptyset)$ to a triangulation $\T_p = \Sigma_p \cap Q_\nu$ such that for a generic character $\chi \in \sigma_p$,  
\[
[(\A^{|\nu|})^{\op{ss}}(\chi) / S_\nu]  \cong \mathcal X_{\Sigma_p}.
\]   
\label{thm: GKZ chamber bijection}
\end{theorem}
\begin{proof}
See Proposition 15.2.9 of \cite{CLS} or Chapter 7, Theorem 1.7 of \cite{GKZ} for the original formulation in terms of polytopes.  The final statement follows from Corollary~\ref{cor: reduce triangulation}.
\end{proof}

Given a set of lattices points $\nu$, we may ask when $S_\nu$ satisfies the quasi-Calabi-Yau condition.  This turns out to be very similar to the notion of a Gorenstein cone.
\begin{definition}
We say that a finite set $\nu = \{v_1, ..., v_n \} \subseteq N$ satisfies the \newterm{Calabi-Yau condition} if there exists a $m \in M$ such that $\langle m, v_i \rangle =1$ for all $i$.
\end{definition}

\begin{example}
If $\nu$ is the set of primitive ray generators of a Gorenstein cone, then it satisfies the Calabi-Yau condition by definition.
\end{example}

\begin{lemma} If the collection $\nu$ satisfies the Calabi-Yau condition, then the corresponding torus $S_\nu$ satisfies the quasi-Calabi-Yau condition.    
\label{lem: CY triangulation}
\end{lemma}

\begin{proof}
The collection $\nu$ gives a map
\[
M \overset{f}{\rightarrow} \Z^{|\nu|} \overset{g} \rightarrow \op{coker} \rightarrow 0
\]
where $f(m) := \sum_{v \in \nu} \langle m, v \rangle e_v$.  Therefore the Calabi-Yau condition on $\nu$ is equivalent to  $\sum_{v \in \nu}  e_v$ lying in the image of $f$ which is the kernel of $g$.

Consider the map
$h : \Z \to \Z^{|\nu|}$ defined by $h(1) = \sum_{v \in \nu}  e_v$.  
Another formulation is of the Calabi-Yau condition on $\nu$ is $g \circ h = 0$.

Recall that $S_\nu = \gm \otimes \op{coker}^\vee$ and the inclusion is given by $\op{Id}_{\gm} \otimes g^\vee$.  Furthermore, under the identification $\gm \otimes \Z^{|\nu|} \cong \gm^{|\nu|}$, the multiplication map is just $\op{Id}_{\gm} \otimes h^\vee$.  So the quasi-Calabi-Yau condition on $S_\nu$ can be formulated as $\op{Id}_{\gm} \otimes (h^\vee \circ g^\vee) = \bf{1}$. 

So, we have reduced to observing that $g \circ h = 0 $ implies $\op{Id}_{\gm} \otimes (h^\vee \circ g^\vee) = \bf{1}$.

\end{proof}

\begin{remark}
As stated in the proof, assuming the quasi-Calabi-Yau condition on $S_\nu$ is equivalent to  $\op{Id}_{\gm} \otimes (h^\vee \circ g^\vee) = \bf{1}$.  This is equivalent to $g \circ h$ being torsion, which means that the quasi-Calabi-Yau condition on $S_\nu$ is slightly less restrictive than the Calabi-Yau condition on the lattice points $\nu$.
\end{remark}

\subsection{Split Toric Vector Bundles}

In this subsection, we recall some basic facts about vector bundles on toric varieties and then give their stacky analogues. We later specialize to the context of nef partitions to prepare ourselves for the proof in the following section.

Let $\Sigma \subseteq N_{\R}$ be a fan and $X_\Sigma$ be the corresponding toric variety. Each ray $\rho \in \Sigma(1)$ gives a torus invariant divisor $D_\rho$ and any Cartier divisor can be expressed as a linear combination, 
\[
D := \sum_{\rho \in \Sigma(1)} a_\rho D_\rho.
\]

We denote by $\mathscr{L}_D$ the total space of $\mathcal{O}_{X_\Sigma}(D)$.  Adding a dilation action to the fibers endows $\mathscr{L}_D$ with the action of an open dense torus.  Indeed $\mathscr{L}_D$ is the toric variety associated to a fan
$\sigma_D \subseteq N_{\R} \times \R$ defined as follows.

   For any cone $\sigma \in \Sigma$, we define $\sigma_D\subseteq N_{\R} \times \R$ by the formula
$$
\sigma_D = \operatorname{Cone}\left( (0,1), (u_\rho, -a_\rho) \middle| \ \rho \in \sigma(1)\right).
$$
The fan $\Sigma_D$ is defined as the collection of cones, $\sigma_D$ for all $\sigma \in \Sigma$, and their proper faces.

\begin{proposition}
The map $\pi: X_{\Sigma_D} \rightarrow X_\Sigma$ induced by the projection $N_{\R} \times \R \rightarrow N_{\R}$ is the usual bundle projection map.
\end{proposition}
\begin{proof}
This is  Proposition 7.3.1 of \cite{CLS}.
\end{proof}

Iterating this construction provides the following description of a direct sum of line bundles. Suppose one has $r$ Cartier divisors $D_i = \sum_{\rho \in \Sigma(1)} a_{i\rho} D_\rho$. Let $e_i$ be the standard basis for $\R^r$. Given $\sigma \in \Sigma$, we define the cone
$$
\sigma_{D_1,\ldots, D_r} : = \operatorname{Cone}\left(u_\rho - a_{1\rho}e_1- \ldots - a_{r\rho}e_r\middle| \  \rho \in \sigma(1)\right) + \op{Cone}(e_i | \ i \in I \subseteq \{1, ..., r\})
$$
and the the fan $\Sigma_{D_1,\ldots, D_r}$ to be the collection of cones $\sigma_{D_1,\ldots, D_r}$.

\begin{remark}
Any toric vector bundle that is also a toric variety is a direct sum of line bundles which can be constructed as above \cite{Oda78}.
\end{remark}

Notice that $S_{\Sigma(1)}  \subseteq \gm^{|\Sigma(1)|}$.  Hence, each ray $\rho \in \Sigma(1)$ gives a character $\chi_\rho$ of $S_{\Sigma(1)}$ via composition with the projection.  Hence, given a divisor $D = \sum a_\rho D_\rho$ on $X_\Sigma$, we can associate a character 
\[
\chi_D := \prod \chi_\rho^{a_\rho}
\]
of $S_\nu$.

\begin{proposition}
Let $D_1, ..., D_r$ be divisors on $X_\Sigma$.  There is an isomorphism of stacks,
\[
\mathcal X_{\Sigma_{D_1, ..., D_r}} \cong \op{tot}(\oplus_{i=1}^r \O_{\mathcal{X}_{\Sigma}}(\chi_{D_i})).
\]

\label{prop: vector bundle on a stack}
\end{proposition}

\begin{proof}
By definition 
\[
\mathcal X_{\Sigma_{D_1, \ldots, D_r}}
=
[X_{\op{Cox}(\Sigma_{D_1, \ldots, D_r})} / S_{\Sigma_{D_1, \ldots, D_r}(1)}].
\]
Now, there is a bijection
\begin{align*}
\Sigma_{D_1, \ldots, D_r}(1) & \to \Sigma(1) \cup \{1, \ldots, r\} \\
u_\rho - \sum a_{i\rho} e_i & \mapsto u_\rho \\
e_i & \mapsto i. \\
\end{align*}
This identifies 
\[
X_{\op{Cox}(\Sigma_{D_1, \ldots, D_r})} \subseteq \A^{\Sigma_{D_1, \ldots, D_r}(1)} = \A^{\Sigma(1)} \times \A^r.
\]

Now, by definition, $\Sigma_{D_1,..., D_r}$ consists of the cones
$$
\sigma_{D_1,\ldots, D_r} := \operatorname{Cone}\left(u_\rho - a_{1\rho}e_1- \ldots - a_{r\rho}e_r\middle| \  \rho \in \sigma(1)\right) + \op{Cone}(e_i | i \in I \subseteq \{1, ..., r\}).
$$

Recall the notation 
\[
\Sigma_r = \{ \op{Cone} (e_i | \ i \in I) | \ I \subseteq \{1, ..., r \} \}.
\]
We have,
\begin{align*}
\op{Cox}(\Sigma_{D_1, ..., D_r}) &  = \{ \op{Cone}(e_\rho | \ \rho \in \sigma) | \ \sigma \in \Sigma_{D_1, ..., D_r}) \} \\
& = \{ \op{Cone}(e_\rho | \ \rho \in \sigma) | \ \sigma \in \Sigma  \} \times \Sigma_r \\
&= \op{Cox}(\Sigma) \times \Sigma_r \\
\end{align*}
The first line is by definition.  The second line comes from the enumeration of cones of $\Sigma_{D_1,..., D_r}$ in terms of cones of $\Sigma$ and subsets of $\{1, ..., r\}$.  The third line is again by definition.

From the above we obtain the formula,
\[
X_{\op{Cox}(\Sigma_{D_1, ..., D_r})} = X_{\op{Cox}(\Sigma)} \times \A^r.
\]

Now set $\nu$ to be the set of primitive ray generators of $\Sigma$ and $\mu$ to be the set of primitive ray generators of $\Sigma_{D_1, ..., D_r}$.
We claim that
\[
S_\nu \cong S_\mu
\]
under which  $S_\nu$ is embedded in $\gm^{|\nu|} \times \gm^r$ where the map to $\gm^{|\nu|}$  is the one that comes with $S_\nu$, and the map to $\gm^r$ is given by the characters $\chi_{D_i}$. 

To prove the claim, first observe that $\nu$ gives a right exact sequence,
\begin{align*}
M & \to \Z^{|\nu|} \to \op{coker} \to 0 \\
m & \mapsto \sum_{\rho \in \Sigma(1)} \langle m, u_\rho \rangle e_\rho.
\end{align*}
Similarly, $\mu$ gives a right exact sequence,
\begin{align*}
M \oplus \Z^r & \to \Z^{|\nu|} \oplus \Z^r \to \op{coker} \to 0 \\
(m, v) & \mapsto \sum_{\rho \in \Sigma(1)} \langle (m,v), u_\rho - \sum a_{i\rho}e_i \rangle e_\rho + \sum_{i=1}^r \langle (m,v), e_i \rangle e_i. \\
\end{align*}
This latter map sends $(m,0)$ to $\sum_{\rho \in \Sigma(1)} \langle m, u_\rho \rangle e_\rho$ and $(0,e_i^*)$ to $-\sum_{\rho \in \Sigma(1)} a_{i\rho} e_\rho + e_i$.  Hence the map
\begin{align*}
\Z^{|\nu|} \oplus \Z^r & \to \Z^{\nu} \\
e_i & \mapsto \sum_{\rho \in \Sigma(1)} a_{i\rho} e_\rho \\
\end{align*}
induces an isomorphism of cokernels.
Applying $\op{Hom}(-,\gm)$ yields the isomorphism $
S_\nu \cong S_\mu$ with the description claimed.

In summary, 
\[
[X_{\op{Cox}(\Sigma_{D_1, ..., D_r})} / S_{\Sigma_{D_1, ..., D_r}(1)}] = [X_{\op{Cox}(\Sigma)} \times \A^r / S_{\Sigma(1)}]
\]
with the action on $\A^r$ given by the characters $\chi_{D_1}, ..., \chi_{D_r}$.  The result follows.

\end{proof}

\begin{lemma}
Let $\Sigma$ be a complete fan and 
\[
D_i = \sum a_{i\rho} D_\rho
\]
be nef divisors.
The dual cone to $|\Sigma_{-D_1, ..., -D_r}|$ is equal to the Cayley cone on the set of polytopes
\[
\Delta_i :=  \{ m \in M_{\R} | \ \langle m, u_\rho \rangle \geq -a_{i\rho} \text{ for all } \rho \in \Sigma(1) \},
\]
i.e.,
\[
|\Sigma_{-D_1, ..., -D_r}|^\vee = \R_{\geq 0} (\Delta_1 * \ldots * \Delta_r) = \R_{\geq 0} (\Delta_1 +e_1^*) + \ldots + \R_{\geq 0} (\Delta_r + e_r^*).
\]

\label{lem: dual cone is Cayley cone}
\end{lemma}
\begin{proof}
First we show that 
\[
 \R_{\geq 0} (\Delta_1 * \ldots * \Delta_r)  \subseteq |\Sigma_{-D_1, ..., -D_r}|^\vee.
\]
For this, it is enough to show that $( \Delta_i +e_i^*) \subseteq |\Sigma_{-D_1, ..., -D_r}|^\vee$.  
Now notice that
\[
\Delta_i = \{ m \in M_{\R} | \ \langle m, u_\rho \rangle \geq -a_{i\rho} \text{ for all } \rho \in \Sigma(1) \}.
\]
Hence, if we pair $m+e_i^* \in \Delta_i$ with a primitive ray generator $e_j$ we get $\delta_{ij} \geq 0$ and if we pair $m+e_i^*$ with a primitive ray generator $u_\rho + \sum_{i=1}^r a_{i\rho}e_i$ we get $\langle m, u_\rho \rangle + a_{i\rho} \geq 0$.

Now we show that   
\[
|\Sigma_{-D_1, ..., -D_r}|^\vee \subseteq  \R_{\geq 0} (\Delta_1 * \ldots * \Delta_r).  
\]
Let $(m, c_1, ..., c_r) \in |\Sigma_{-D_1, ..., -D_r}|^\vee$.  As $e_i \in |\Sigma_{-D_1, ..., -D_r}|$, we have an inequality 
\[
\langle (m,c_1, ..., c_r), e_i \rangle = c_i \geq 0.
\] 
Similarly, $(u_\rho, a_{1\rho}, ..., a_{r\rho}) \in |\Sigma_{-D_1, ..., -D_r}|$ implies that
\[
\langle (m,c_1, ..., c_r), (u_\rho, a_{1\rho}, ..., a_{r\rho})  \rangle = \langle m, u_\rho \rangle + \sum c_i a_{i\rho} \geq 0.
\] 
Hence, $m$ lies in the polytope
\[
\Delta_{\sum c_i D_i} := \{ m \in M_{\R} | \ \langle m, u_\rho \rangle \geq - \sum c_i a_{i\rho} \text{ for all } \rho \in \Sigma(1) \}.
\]
Now, since each $D_i$ is nef, Lemma 6.16 and Theorem 6.17(g) of \cite{CLS}, imply that
\[
\Delta_i = \{ m \in M_{\R} \ | \ \op{min}_{d \in \Delta_i} \langle d, u \rangle \leq \langle m, u \rangle \}.
\]
and
\[
\Delta_{\sum c_i D_i} := \{ m \in M_{\R}  \ | \  \op{min}_{d \in  \Delta_{\sum c_i D_i}} \langle d, u \rangle \leq \langle m, u \rangle \}.
\]

Hence, by Theorem A.18 of \cite{Oda88}, 
\[
\Delta_{\sum c_i D_i} = \sum c_i \Delta_{D_i}.
\]
Therefore, $m \in \sum c_i \Delta_{D_i}$.  So, we may write
\[
m = \sum m_i
\]
for some $m_i \in c_i \Delta_{D_i}$.
Then, 
\[
(m, c_1, ..., c_r) = \sum (m_i +c_ie_i^*)  \in \R_{\geq 0} (\Delta_1 +e_1^*) + \ldots + \R_{\geq 0} (\Delta_r + e_r^*).
\]

\end{proof}

\begin{remark}
The lemma above is essentially the special case for toric varieties of the statement that global functions on a vector bundle $\mathcal E$ are given by
\[
\bigoplus_{i \in \N} \op{H}^0(\op{Sym}^i(\mathcal E^\vee)).
\]
\end{remark}

\begin{lemma}
Let  $\Sigma$ be a fan and suppose $X_\Sigma$ is semiprojective.  Let $D$ be a nef divisor.  Then  $\op{tot} \O(-D)$ is semiprojective.
\label{lem: semiprojective}
\end{lemma}
\begin{proof}

By Proposition 7.2.9 of \cite{CLS}, $\Sigma$ has full dimensional convex support.
The support of $\Sigma_{-D}$ is
\[
|\Sigma_{-D}| =  \{ (u, \lambda) | \lambda \geq \phi_{-D}(u) \} = \{(u, \lambda) | \lambda \leq \phi_{D}(u) \} 
\]
where $\phi_{-D}$ is the support function corresponding to $-D$ (see e.g. page 335 of \cite{CLS}).  By Theorem 6.1.7 of loc.\ cit.\ $\phi_D$ is convex since being globally generated is equivalent to being nef when $\Sigma$ has full-dimensional convex support (Theorem 6.3.12 of loc. cit.).
Hence given $(u, \lambda), (u, \lambda') \in |\Sigma_{-D}|$, we have 
\begin{align*}
t\lambda + (1-t)\lambda' & \leq t\phi_D(u) + (1-t)\phi_D(u') \\
& \leq \phi_D(tu+(1-t)u').
\end{align*}
Therefore $|\Sigma_{-D}|$ is convex.

By Proposition 7.2.9 of \cite{CLS}   $X_{\Sigma_{-D}}$ is  semiprojective if and only if $X_{\Sigma_{-D}}$ is quasi-projective and $\Sigma_{-D}$ has full dimensional convex support.
Since $X_{\Sigma_{-D}}$ is a vector bundle over a quasi-projective variety, the pullback of any ample line bundle is ample. Therefore $X_{\Sigma_{-D}}$ is quasi-projective.
Furthermore, let $\sigma$ be a maximal cone in $\Sigma$ (which exists since $X$ is semiprojective).  Then the cone $\op{Cone}(u_\rho - \sum a_{i\rho} | \ \rho \in \sigma(1) ) + \op{Cone}(e_1, ..., e_r) \subseteq (N \oplus \Z^r)_{\R}$ and therefore has full dimensional support.

\end{proof}

We now specialize to the case of a nef-partition i.e. let $D_i \neq 0$ be nef divisors on $X_{\Sigma}$ such that
\[
D_1 + \ldots + D_r = -K := -\sum_{\rho \in \Sigma(1)} D_\rho
\]
as subvarieties (not just up to linear equivalence).
In the case where $\Sigma$ is the normal fan to a reflexive polytope $\Delta$, recall that a nef partition can be equivalently characterized as $p_1 \in \Delta_1, \ldots, p_r \in \Delta_r$ such that $\Delta_1 + \ldots + \Delta_r = \Delta$ and $\sum_{i=1}^r p_i$ is the unique interior lattice of $\Delta$.

\begin{proposition}
Let $\Sigma$ be a simplicial fan and suppose $X_\Sigma$ is semiprojective.  Let $D_i = \sum_{\rho \in \Sigma(1)} a_{i \rho} D_\rho$ be nef-divisors for $i=1, \ldots, r$.
Define
\[
\nu_{-D_1, ...,  -D_r} := \{ (u_{\rho} + \sum a_{i\rho}e_i) | \ \rho \in \Sigma(1), 1 \leq i \leq r \} \cup \{e_1, ..., e_r \}\subseteq N_{\R} \times \R^r.
\]
Then $\Sigma_{-D_1,.., -D_r} \cap \op{Conv}(\nu_{-D_1, ..., -D_r})$ is a regular triangulation of $\nu_{-D_1,.., -D_r}$.
\label{prop: regular triangulation vector bundle}
\end{proposition}
\begin{proof}
First, $|\Sigma_{-D_1,.., -D_r}| = \op{Cone}(\nu_{-D_1, ..., -D_r})$ by definition.  
Second, $X_{\Sigma_{-D_1,.., -D_r}}$ is semiprojective by iterating Lemma~\ref{lem: semiprojective}.
Third, the rays of $\Sigma_{-D_1,.., -D_r}$ are precisely the elements of $\nu$, also by definition.  

Therefore the fan $\Sigma_{-D_1,.., -D_r}$ and the set $I_\emptyset = \emptyset$ satisfy the conditions of Proposition 14.4.1 of \cite{CLS}. Hence, we get a chamber of the GKZ fan corresponding to the GKZ cone associated to this fan and set.

By Theorem~\ref{thm: GKZ chamber bijection}, $\Sigma_{-D_1,.., -D_r} \cap \op{Conv}(\nu_{-D_1, ..., -D_r})$ is a regular triangulation of $\nu$.
\end{proof}

Since the above proposition requires $\Sigma$ to be simplicial, if we want to apply it more generally, we will need to replace an arbitrary fan $\Sigma^{\op{pre}}$ with a simplicial refinement. 

\begin{proposition}
\label{prop: simplicialization}
Every fan $\Sigma^{\op{pre}}$ has a refinement $\Sigma$ with the following properties:
\begin{enumerate}
\item $\Sigma$ is simplicial.
\item $\Sigma^{\op{pre}}(1) = \Sigma(1)$.
\item $\Sigma$ contains every simplicial cone of $\Sigma^{\op{pre}}$.
\item $\Sigma$ is obtained from $\Sigma^{\op{pre}}$ by a sequence of star subdivisions.
\item The induced toric morphism $X_\Sigma \to X_{\Sigma^{\op{pre}}}$ is projective.
\end{enumerate}
\end{proposition}
\begin{proof}
This is Proposition 11.1.7 of \cite{CLS}.
\end{proof}

\begin{lemma}
Let $K := -\sum_{\rho \in \Sigma(1)} D_\rho$ and $D_i \neq 0$ be effective divisors on $X_{\Sigma}$ such that $D_1 + ... + D_r = -K$ as subvarieties.   Then $|\Sigma_{-D_1,\ldots, -D_r}|$ is a Gorenstein cone and
the primitive ray generators of $\Sigma_{-D_1, ..., -D_r}$ lie in 
$|\Sigma_{-D_1, ..., -D_r}|_{(1)}$.
\label{lem: CY property for nef partition}
\end{lemma}

\begin{proof}
By definition the ray generators are
\[
\Sigma_{-D_1, ..., -D_r}(1) = \{e_1, ..., e_r\} \cup \{u_\rho + \sum_{i=1}^r a_{i\rho}e_i | \ \rho \in \Sigma\} \subseteq (N \times \Z^r)_{\R}.
\]
Now, since by assumption,
\[
D_1 + ... + D_n = -K =  \sum_{\rho \in \Sigma}D_\rho
\] 
we see that for each $\rho$ there is a unique $i$ such that $a_{i\rho}  = 1$  and $a_{i\rho} = 0$ otherwise.  Since the $u_\rho$ are primitive, this implies the ray generators above are also primitive.  Now simply notice that setting $\op{deg}^\vee := \sum_{i=1}^r e_i^*$ the above observation is equivalent to the fact that all the ray generators lie in the affine hyperplane $\langle \op{deg}^\vee, - \rangle =1$.  Hence $|\Sigma_{-D_1, ..., -D_r}|$ is Gorenstein and the primitive ray generators of $\Sigma_{-D_1, ..., -D_r}$ lie in 
$|\Sigma_{-D_1, ..., -D_r}|_{(1)}$.
\end{proof}

\begin{proposition}
Let $\Delta$ be a reflexive polytope and $\Sigma$ be its normal fan.  Consider a nef partition, $\Delta_1, ..., \Delta_r$ and the associated divisors $D_i$ given by the sections corresponding to $p_i \in \Delta_i$.
The cone $|\Sigma_{-D_1, \ldots, -D_r}|$ is a completely split reflexive Gorenstein cone of index $r$ whose dual cone is the Cayley cone of $\Delta_1, ..., \Delta_r$ and the primitive ray generators of the fan $\Sigma_{-D_1, \ldots, -D_r}$ lie in $|\Sigma_{-D_1, \ldots, -D_r}|_{(1)}$.
\label{prop: summary}
\end{proposition}
\begin{proof}
This is just an assembly of previous results.
By Lemma~\ref{lem: dual cone is Cayley cone},  $|\Sigma_{-D_1, \ldots, -D_r}|$ is the dual cone to the Cayley cone of $\Delta_1, \ldots, \Delta_r$.  By Proposition~\ref{prop: BN nef}, it follows that $|\Sigma_{-D_1, \ldots, -D_r}|$ is a completely split reflexive Gorenstein cone of index $r$.
By Lemma~\ref{lem: CY property for nef partition}, the primitive ray generators lie in $|\Sigma_{-D_1, \ldots, -D_r}|_{(1)}$.
\end{proof}

 \section{Proof of the Batyrev-Nill Conjecture}

\label{sec: BN}
Let $\sigma \subseteq \overline M_{\R}$ be a completely-split reflexive Gorenstein cone of index $r$.  Furthermore, suppose we have multiple splittings,
\[
\sigma = \R_{\geq 0}(\Delta_1 * \ldots * \Delta_r) = \R_{\geq 0}(\Delta_1' * \ldots * \Delta_r').
\]

\begin{remark}
 As a warning to the reader, we reverse the roles of $\Delta_i$ and $\nabla_i$ in the original notation.
 That is, following the setup of \cite{BN07}, we would have have an isomorphism
 \[
 \R_{\geq 0}(\nabla_1 * \ldots * \nabla_r) = \R_{\geq 0}(\nabla_1' * \ldots * \nabla_r')
\]
instead of our assumption that
\[
 \R_{\geq 0}(\Delta_1 * \ldots * \Delta_r) = \R_{\geq 0}(\Delta_1' * \ldots * \Delta_r').
\]
The statement of results does not require both nef-partitions,  hence the reversal.
\end{remark}

By Proposition~\ref{prop: Cayley comparison}, the degree element can be simultaneous expressed as a sum of nonzero elements,
\[
\op{deg}^\vee = \sum_{i=1}^r e_i = \sum_{i=1}^r e_i'.
\]
where 
\[
\Delta_i := \{x \in \sigma_{(1)} | \  \langle x, e_j \rangle =0 \text{ for } j \neq i \} 
\]
and similarly
\[
\Delta_i' := \{x \in \sigma_{(1)} | \  \langle x, e_j' \rangle =0 \text{ for } j \neq i \}.
\]

Let 
\[
w := \sum_{m \in \sigma_{(1)}} c_m x^m.
\]
Since
\[
\sigma_{(1)} = \Delta_1 * \ldots * \Delta_r = \Delta_1' * \ldots * \Delta_r',
\]
we can write
\[
w = \sum_{i=1}^r w_i = \sum_{i=1}^r w_i'
\]
where 
\[
w_i := \sum_{m \in \Delta_i}  c_m x^m \tand w_i' := \sum_{m \in \Delta_i'}  c_m x^m.
\]
We set the Minkowski sums of polytopes
\[
\Delta := \Delta_1 + \ldots + \Delta_r \tand \Delta' := \Delta_1' + \ldots + \Delta_r'
\]

We may view $w_i$ as a section of $\O_{\mathcal X_{\Delta}}(\chi_{D_i})$ where $D_i$ is the divisor associated to $\Delta_i$ and similarly view $w_i'$ as a section of $\O_{\mathcal X_{\Delta'}}(\chi_{D_i'})$ where $D_i'$ is the divisor associated to $\Delta_i'$.

Let $\Sigma$ (respectively $\Sigma'$) be a simplicial refinement of the normal fan to $\Delta$ (respectively $\Delta'$) satisfying the properties of Proposition~\ref{prop: simplicialization} (this corresponds to Batyrev's MPCP desingularization in the introduction).
Let $Z$ be the closed substack of  $\mathcal X_{\Sigma}$ defined by the common zeros of the pullbacks of $w_i$ and $Z'$ be the closed substack  of $\mathcal X_{\Sigma'}$ defined by the common zeros of the pullbacks of the $w_i'$.

\begin{remark}
If $X_\Sigma$ is smooth then by Theorem~\ref{thm: stack realization},
\[
\mathcal X_\Sigma \cong X_\Sigma
\]
and $Z$ is the closed subvariety of $X_\Sigma$ defined by the common zeros of the pullbacks of the $w_i$.  Similarly, if $X_\Sigma'$ is smooth, then $Z'$ is the closed subvariety of $X_{\Sigma'}$ defined by the common zeros of the pullbacks of the $w_i'$.
\end{remark}

The following theorem is the stacky interpretation of Conjecture 5.3 of Batyrev and Nill in \cite{BN07}, following the clarification of this conjecture in Remark 5.4 of loc. cit.

\begin{theorem}\label{BNconjecturetheorem}
Assume that $Z,Z'$ have the expected dimension  $\op{dim }\sigma - 2r.$
  Then there is an equivalence of categories
\[
\dbcoh{Z} \cong \dbcoh{Z'}.
\]

\end{theorem}

\begin{proof}
Take the set $\nu = N \cap \sigma^\vee_{(1)}$.  The set $\nu$ manifestly satisfies the Calabi-Yau condition, hence satisfies the quasi-Calabi-Yau condition for $S_\nu$ by Lemma~\ref{lem: CY triangulation}.  

Let $\Sigma$ (respectively $\Sigma'$) be a simplicial refinement of the normal fan to $\Delta$ (respectively $\Delta'$) satisfying the properties of Proposition~\ref{prop: simplicialization}.  By Propositions~\ref{prop: regular triangulation vector bundle} and~\ref{prop: summary}, the collection $\Sigma_{-D_1, \ldots,-D_r}$ provides a regular triangulation of $\sigma^\vee$ with vertices in $\nu$.  Similarly, the collection $\Sigma'_{-D_1', \ldots,-D_r'}$ where $\Sigma'$ is the normal fan to $\Delta'$ provides another regular triangulation of $\sigma^\vee$ with vertices in $\nu$.

By Theorem~\ref{thm: GKZ chamber bijection}, these two regular triangulations correspond to chambers of the GKZ fan for the action of $S_\nu$ on $\A^{|\nu|}$ with Cox open sets which we enumerate by $U_p,U_q$ respectively so that by Proposition~\ref{prop: vector bundle on a stack},
\[
[U_p / S_\nu] = \mathcal X_{\Sigma_{-D_1, \ldots, -D_r}} = \op{tot}(\oplus_{i=1}^r \O_{\mathcal{X}_{\Sigma}}(\chi_{-D_i}))
\]
and 
\[
[U_q / S_\nu] = \mathcal X_{\Sigma'_{-D_1', \ldots, -D_r'}} = \op{tot}(\oplus_{i=1}^r \O_{\mathcal X_\Sigma'}(\chi_{-D_i'})).
\]

Also notice that $w$ is an $S_\nu$ invariant function since it descends to sections $\oplus w_i \in \oplus \op{H}^0(\O(D_i))$ and $\oplus w_i' \in \oplus \op{H}^0(\O(D_i'))$ respectively.

Now we add $R$-charge.  This is a $\gm$ action on $\A^{|\nu|}$, which for $x_v$ and $s \in \gm$ is given explicitly as
\begin{equation}
s \cdot x_v := 
\begin{cases}
s x_v & \tif v \in \{e_1, \ldots, e_r\}\\
x_v & \tif v \notin \{e_1, \ldots, e_r\}.\\
\end{cases}
\label{eq: Rcharge}
\end{equation}
This $R$-charge is the fiberwise dilation for $\mathcal X_{\Sigma_{-D_1, \ldots, -D_r}}$ under the isomorphism with $[U_p/S_{\nu}]$, i.e., there is an isomorphism 
\begin{equation}
[U_p/S_{\nu} \times \gm] \cong [\mathcal X_{\Sigma_{-D_1, \ldots, -D_r}}/ \gm]
\label{eq: realization}
\end{equation}
where $\gm$ acts by fiberwise dilation.

We now check that this can be realized as fiberwise dilation on the quotient 
\[
[U_q/S_{\nu} \times \gm] \cong [\mathcal X_{\Sigma'_{-D_1', \ldots, -D_r'}}/ \gm].
\]
This is not immediate and requires the use of an automorphism of $S_{\nu} \times \gm$.

Without loss of generality, we may reorder our sets so that 
\[
\{e_1, \ldots, e_r\} \cap \{e_1', \ldots, e_r'\} = \{e_1, \ldots, e_c \} = \{e_1', \ldots, e_c' \}
\]
  Then we have an equality
\begin{equation}
\sum_{i =c+1}^{r} e_i = \sum_{i =c+1}^{r} e_i'.   
\label{eq: Cayley relation}
\end{equation}

Also consider the  $\gm$ acting on $\A^{\nu}$ so that for $v \in \nu$, $s \in \gm$ acts by
\begin{equation}
s \cdot x_v := 
\begin{cases}
\frac{1}{s}x_v & \tif v = e_i \text{ for } c+1 \leq i \leq r\\
s x_v & \tif v = e_i' \text{ for } c+1 \leq  i  \leq r \\
x_v & \text{otherwise}\\
\end{cases}
\label{eq: fix Rcharge}
\end{equation}
We claim that $\gm \subseteq S_\nu$. By definition, $S_\nu$ lies in an exact sequence
\[
0 \longrightarrow S_\nu \overset{\widehat{\pi}}{\longrightarrow} \gm^{n} \overset{\widehat{f_{\nu}}}{\longrightarrow} \gm^d
\]
and the $\gm$ above lies in $\gm^{n}$, so it is enough to check that $\widehat{f_\nu}(\gm)$ is trivial (see Equation~\eqref{eq: f} for a definition). 
This follows from the vanishing of  
\[
f_\nu ( \sum_{i=c+1}^r e_{e_i} - \sum_{i=c+1}^r e_{e_i'}  ) = \sum_{i=c+1}^r e_i - \sum_{i=c+1}^r e_i' = 0
\]
since after applying $\op{Hom}(- , \gm)$ the vector $\sum_{i=c+1}^r e_{e_i} -  \sum_{i=c+1}^r e_{e_i'} $ is the generator of $\gm$.

Now notice that $S_\nu \subseteq \gm^{\nu}$ so that projection onto the factor corresponding to $v \in \nu $ gives a character $\chi_v$.
Setting $v = e_r'$ splits the above $\gm \subseteq S_\nu$.

Let $S_\nu = \overline{S_\nu} \times \gm$ using the above splitting.   There is an automorphism \begin{align*}
F : (\overline{S_\nu} \times \gm) \times \gm & \to (\overline{S_\nu} \times \gm) \times \gm \\
(\bar{s}, s_1, s_2) & \mapsto (\bar{s}, s_2 s_1, s_2)). 
\end{align*}

Under $F$, the action of $S_\nu \times 1$ on $\A^{|\nu|}$ is the same as $F(S_\nu \times 1 ) = S_\nu \times 1$.  However, the projection action of $1 \times \gm$, becomes the action of the element $F(1,1,s) = (1,s,s)$.

The action of $(1,t,1)$ is given by Equation~\eqref{eq: fix Rcharge} and the action of $(1,1,s)$ is given by Equation~\eqref{eq: Rcharge}.  Combining these two equations we get:
\begin{equation}
(1,s,s) \cdot x_v := 
\begin{cases}
s x_v & \tif v \in \{e_1', \ldots, e_r'\}\\
 x_v & \tif v \notin \{e_1', \ldots, e_r'\}\\
\end{cases}
\end{equation}
This is the $R$-charge on $X_{\Sigma_{D_1', \ldots, D_r'}} = [U_q / S_\nu]$ i.e., using the automorphism $F$, we have realized an isomorphism 
\begin{equation}
[U_q/S_{\nu} \times \gm] \cong [\mathcal X_{\Sigma_{-D_1', \ldots, -D_r'}}/ \gm]
\label{eq: realization'}
\end{equation}
where $\gm$ acts by fiberwise dilation. We now have \begin{align*}
\dbcoh{Z} & \cong \dsing{[Z_p/S_\nu \times \gm]} \\
& \cong \dsing{[Z_q/S_\nu \times \gm]} \\ 
& \cong \dbcoh{Z'}
\end{align*}
where the first line is Theorem~\ref{thm: isik-shipman}  (uses  Equation~\eqref{eq: realization} and the assumption that $Z$ has the expected dimension), the second line is Theorem~\ref{thm: HW}, and the third line is Theorem~\ref{thm: isik-shipman} again (uses Equation~\eqref{eq: realization'} and the assumption that $Z'$ has the expected dimension).
\end{proof}

\begin{remark}
It does not follow from the above result that $Z$ and $Z'$ are birational.  While, $Z$ and $Z'$ are both open substacks of the critical locus of $w$ modulo $S_\nu$, they lie on different irreducible components.  Indeed, $Z \subseteq Z( x_{e_1}, \ldots, x_{e_r} )$ while $Z' \subseteq Z(x_{e_1'}, \ldots, x_{e_r'})$.  Moreover, $Z( x_{e_1}, \ldots, x_{e_r} )$ and $Z(x_{e_1'}, \ldots, x_{e_r'})$ may be unstable loci which are exchanged when varying GIT quotients for this action of $S_\nu$.  
\label{rem: not birational}
\end{remark}

\begin{remark}
Assuming only that $w_i \neq 0$, the theorem still holds if we replace $Z,Z'$ by their derived intersections in the sense of derived algebraic geometry.  Namely, the Koszul complex on $w_i$ and $w_i'$ respectively form differential graded schemes on $\mathcal X_{\Delta}, \mathcal X_{\Delta'}$ and the derived categories of these differential graded schemes are equivalent, see Remark 3.7 of \cite{Isik}.
\label{rem: dg scheme}
\end{remark}

\end{document}